\documentclass[a4paper,11pt]{article}

\usepackage{fullpage,authblk}
\usepackage{amsmath,amsthm,amssymb}
\newtheorem{theorem}{Theorem}
\newtheorem{lemma}[theorem]{Lemma}
\newtheorem{conj}{Conjecture}
\newcommand{\N}{\mathbb N}

\title{Countable graphs are majority $3$-choosable}
\author{John Haslegrave}
\affil{Mathematics Institute, University of Warwick, CV4 7AL, UK}

\begin{document}
\maketitle

\begin{abstract}The Unfriendly Partition Conjecture posits that every countable graph admits a $2$-colouring in which for each vertex there are at least as many bichromatic edges containing that vertex as monochromatic ones. This is not known in general, but it is known that a $3$-colouring with this property always exists. Anholcer, Bosek and Grytczuk recently gave a list-colouring version of this conjecture, and proved that such a colouring exists for lists of size $4$. We improve their result to lists of size $3$; the proof extends to directed acyclic graphs. We also discuss some generalisations.
\end{abstract}

\section{Introduction}
It is a simple exercise to show that the vertices of any finite graph can be partitioned into two parts so that every vertex is in the opposite part to at least half of its neighbours. This was first observed by Lov\'asz \cite{Lov}. Such a partition is often referred to as ``unfriendly''. A natural question is whether infinite graphs necessarily have unfriendly partition, where if a vertex has infinite neighbourhood we interpret ``at least half'' to mean a set of the same cardinality as the whole neighbourhood. Shelah and Milner \cite{SM90} answer this question (which they attribute to Cowan and Emerson) in the negative by constructing uncountable counterexamples; however, they conjecture that any countable graph has an unfriendly partition.

This conjecture has been proved in some cases, such as for graphs with finitely many vertices of infinite degree by Aharoni, Milner and Prikry \cite{AMP}, for rayless graphs by Bruhn, Diestel, Georgakopoulos and Spr\"ussel \cite{BDGS}, and for graphs with no subdivision of an infinite clique by Berger \cite{Ber}; the first two results mentioned make no assumption on the cardinality of the graph.

Unfriendly partitions may be rephrased in the language of colourings. A \textit{majority colouring} of a graph is an assignment of colours to vertices such that at most half of the edges incident with any vertex are monochromatic, and a graph is \textit{majority $\ell$-colourable} if it has a majority colouring using at most $\ell$ colours. The Unfriendly Partition Conjecture is that every countable graph is majority $2$-colourable. Shelah and Milner \cite{SM90} showed that every graph (even without the assumption of countability) is majority $3$-colourable.

A classical extension of (proper) colouring of graphs is the concept of list colouring, introduced independently by Vizing \cite{Viz} and by Erd\H{o}s, Rubin and Taylor \cite{ERT}. Instead of assigning colours to vertices from a fixed palette of $\ell$ colours, each vertex $v$ must be assigned one of a list $L(v)$ of $\ell$ colours. Does a suitable colouring exist for every possible system of lists? While it is natural to suppose that this is hardest to achieve when all lists coincide, in fact this is not the case. For example, $K_{3,3}$ can be properly $2$-coloured, but there is a system of lists of size $2$ from which no proper colouring exists. As well as being an interesting problem in its own right, moving to the more general setting of list colouring can facilitate proving results about colouring. More recently, the even more general setting of correspondence (or DP) colouring by Dvo\v{r}\'ak and Postle \cite{DP18} has attracted a great deal of interest. 

Majority list colourings were introduced as an open question by Kreutzer, Oum, Seymour, van der Zypen and Wood \cite{KOSvW}, with the first significant results being due to Anholcer, Bosek and Grytczuk \cite{ABG17}. A graph $G$ is \textit{majority $\ell$-choosable} if for any system of lists $\bigl(L(v)\bigr)_{v\in V(G)}$ of size $\ell$ there is a majority colouring of $G$ in which each vertex $v$ is assigned a colour from $L(v)$. Anholcer, Bosek and Grytczuk \cite{ABG20} recently made the following analogue of the Unfriendly Partition Conjecture.
\begin{conj}\label{c:simple}Any countable graph is majority $2$-choosable.\end{conj}
Kreutzer, Oum, Seymour, van der Zypen and Wood \cite{KOSvW} also extended the concept of majority colourings to digraphs, by requiring that at most half of the outgoing edges from any vertex are monochromatic. In this case three colours are necessary even for some finite digraphs. Anholcer, Bosek and Grytczuk \cite{ABG20} make the following conjectures for countable digraphs.
\begin{conj}\label{c:directed}Any countable digraph is majority $3$-choosable.\end{conj}
\begin{conj}\label{c:acyclic}Any countable acyclic digraph is majority $2$-choosable.\end{conj}
They make progress towards these conjectures by proving that every countable graph and every countable digraph is majority $4$-choosable. We make further progress towards Conjectures \ref{c:simple} and \ref{c:acyclic} by showing that every countable graph and every countable acyclic digraph is majority $3$-choosable. In fact our proofs work in the more general setting of majority correspondence colouring, which we may define in the spirit of Dvo\v{r}\'ak and Postle \cite{DP18}. However, for ease of reading we present our main results for majority choosability, and defer remarks on the correspondence setting to the end of the paper.

\section{Main results}
We will use the following (slightly modified) lemma of Anholcer, Bosek and Grytczuk \cite{ABG20}; we give a short proof for completeness.
\begin{lemma}\label{restrict}Let $V$ be a countable set and $\mathcal X$ be a countable (possibly finite) collection of infinite subsets of $V$. Suppose that each $v\in V$ has a list $L(v)$ of $\ell+1$ colours. Then there is a choice of $\ell$-element sublists $L'(v)\subset L(v)$ such that for every colour $c$ and set $X\in\mathcal X$ there are infinitely many $v\in X$ such that $c\not\in L'(v)$.\end{lemma}
\begin{proof}Note that the set $C:=\bigcup_{v\in V}L(v)$ of all colours is countable. Fix an ordering of the countable set $\mathcal X\times C\times \N$. For each triple $(X,c,i)$ in turn, choose any $v\in X$ which has not previously been chosen (which is possible since $X$ is infinite). If $c\in L(v)$, set $L'(v)=L(v)\setminus \{c\}$. Finally, arbitrarily choose any sublist $L'(v)$ which has not previously been defined.

Note that for every $(X,c)$ every element $v$ that was chosen for a triple of the form $(X,c,i)$ has the property that $v\in X$ and $c\not\in L'(v)$. There are infinitely many such elements, and so these sublists have the required property.
\end{proof}
In the proof of our main result we use the notation $N(v)$ and $N[v]$ for the open and closed neighbourhoods respectively of a vertex $v$.
\begin{theorem}\label{t:simple}Every countable graph is majority $3$-choosable.\end{theorem}
\begin{proof}Apply Lemma \ref{restrict} with $\ell=2$, where $\mathcal X=\{N(v_i)\mid d(v_i)=\infty\}$, to obtain a system of lists $\bigl(L'(v)\bigr)_{v\in V(G)}$. Order the vertices $v_1,v_2,\ldots,$ and for each $n\in\N$ consider the subgraph $G_n$ induced by $v_1,\ldots,v_n$. Fix a colouring $\chi_n$ Since it is finite, $G_n$ is majority $2$-choosable: any colouring which minimises the total number of monochromatic edges is a majority colouring. In particular, there is a colouring $\chi_n$ which colours $v_i$ from $L'(v_i)$ for each $i\leq n$ which is a majority colouring for $G_n$.

We now use a compactness argument. Infinitely many of the colourings $(\chi_n)_{n\in\N}$ agree on the colour of $v_1$; let $\chi(v_1)$ be this colour. Of these, infinitely many also agree on the colouring of $v_2$; let $\chi(v_2)$ be this colour. Continuing in this manner we get a colouring $\chi$ of $V(G)$ in which $v_i$ receives a colour from $L'(v_i)$ for each $i$, and such that for each $n\in N$ there exists $m\geq n$ such that $\chi(v_i)=\chi_m(v_i)$ for every $i\leq n$.

We claim that $\chi$ is a majority colouring. Indeed, if $v_i$ has finite degree then, writing $n=\max\{j:v_j\in N[v_i]\}$, there exists $m\geq n$ such that $\chi_m(v_j)=\chi(v_j)$ for every $j\leq n$. In particular, these two colourings agree on $v_i$ and all its neighbours, and since $\chi_m$ is a majority colouring on $G_m$, at most half of the edges containing $v_i$ are monochromatic in $\chi_m$ and hence $\chi$. Alternatively, if $v_i$ has infinite degree then by taking $X=N(v_i)$ and $c=\chi(v_i)$ in the definition of the sublists $L'(v_j)$, $v_i$ has infinitely many neighbours $v_j$ for which $\chi(v_i)\not\in L'(v_j)$, and so $\chi(v_j)\neq\chi(v_i)$. Thus the conditions of a majority colouring are satisfied at every vertex.
\end{proof}
This proof does not give a good bound for digraphs, where it is an open conjecture to even show that finite digraphs are majority $3$-colourable \cite{KOSvW}. However, in the much simpler case of acyclic digraphs we can use this method. Any finite acyclic digraph is majority $2$-choosable, since we may colour vertices in reverse topological ordering (so that each vertex is processed after all its outneighbours), giving each vertex the colour from its list which is less common among its neighbours. The proof above, replacing neighbourhoods and degrees by outneighbourhoods and outdegrees throughout, therefore gives the following bound.
\begin{theorem}\label{t:acyclic}Every countable acyclic digraph is majority $3$-choosable.\end{theorem}
This strengthens an earlier result of Anholcer, Bosek and Grytczuk \cite{ABG19} that countable acyclic digraphs are majority $3$-colourable.
\section{Comparison with the result of Shelah and Milner}
Suppose we are primarily interested in majority $3$-colourability of countable graphs. This is the intersection of our result and that of Shelah and Milner \cite{SM90}, and so it is natural to compare the two proofs for that case. From \cite{SM90} we can extract a proof of majority $3$-colourability for countable graphs which is substantially shorter than that of their full result, and which we describe informally as follows.

\medskip
For a countable graph $G$ on vertex set $V$, let $B_0$ be the set of all vertices of finite degree. We recursively define $B_\alpha$ for ordinals $\alpha$ as follows: $B_\alpha$ is the set of all vertices not in $\bigcup_{\beta<\alpha}B_\beta$, but with infinitely many neighbours in that set. Stop this process at the first ordinal $\gamma$ such that $B_\gamma$ is empty, and let $C:=V\setminus\bigcup_{\beta<\gamma}B_\beta$ be the set of leftover vertices. Note that in the induced subgraph $G[C]$ all vertices have infinite degree, so (by a result of Aharoni, Milner and Prikry \cite{AMP}) it has a majority $2$-colouring, which we use to define colours of vertices in $C$. For each $\alpha>1$, since every vertex in $B_\alpha$ has infinitely many neighbours in $\bigcup_{1\leq\beta<\alpha}B_\beta$, it is sufficient to colour it differently from infinitely many neighbours in this set. We may do this inductively: colour $B_1$ with a single colour $c$ and for each $\alpha>1$ colour each vertex in $B_\alpha$ differently to the majority of its neighbours in $\bigcup_{1\leq\beta<\alpha}B_\beta$. Finally, we colour vertices in $B_0$. By a compactness argument, for any given colouring of $V\setminus B_0$, we can $2$-colour the vertices in $B_0$ such that each vertex has at most half its edges (in $G$) monochromatic. We may choose such a colouring using only the two colours other than $c$, and this ensures that each vertex in $B_1$ has infinitely many neighbours in $B_0$ of a different colour.

\medskip
Our proof of Theorem \ref{t:simple} is perhaps simpler than this, and although their proof has the significant advantage that it can be generalised to the uncountable case, ours has the more modest advantage that it gives majority $3$-choosability rather than just majority $3$-colourability. Both of these advantages are genuine. The proof of Shelah and Milner relies on the fact that we are colouring from a fixed palette, rather than list colouring, to give every vertex in $B_1$ the same colour. However, Lemma \ref{restrict} relies on $V$ and $\mathcal X$ being countable; if they may be uncountable then we may choose $\mathcal X$ to be all infinite subsets of a given countable subset of $V$, and the result will not hold. Nevertheless, we conjecture that the union of the two results is true.
\begin{conj}Every graph is majority $3$-choosable.\end{conj}
\section{Extensions}
\subsection{$1/k$-majority colouring}
Kreutzer, Oum, Seymour, van der Zypen and Wood \cite{KOSvW} generalised majority colouring (in the setting of digraphs) by asking how many colours are needed to colour a finite digraph such that every vertex has the same colour as at most a $1/k$ proportion of its outneighbours, where $k\in\N$. They call such a colouring a \textit{$(1/k)$-majority colouring}. Gir\~ao, Kittipassorn and Popielarz \cite{GKP} observed that at least $2k-1$ colours may be required, and proved that $2k$ colours is sufficient. In fact they showed this in the more general list-colouring setting, that is, every finite digraph is $(1/k)$-majority $2k$-choosable, and Knox and \v{S}\'amal \cite{KS18} independently proved the same result.

For finite simple graphs, the corresponding notion is much more straightforward: it is easy to see that for any finite graph and any system of lists of size $k$, there is a colouring $\chi$ for which each vertex is the same colour as at most a $1/k$ proportion of its neighbours, and in fact each vertex $v$ has at least as many neighbours of colour $c$ as colour $\chi(v)$ for each $c\in L(v)$ (any colouring minimising the number of monochromatic edges has this property). Even more simply, any finite acyclic digraph is $(1/k)$-majority $k$-choosable.

Our methods give corresponding, but weaker, results for countable graphs in all of these settings. Suppose that any finite induced subgraph of a countable graph (or digraph) $G$ is $(1/k)$-majority $\ell$-choosable. Applying Lemma \ref{restrict} with this value of $\ell$, and mimicking the proof of Theorems \ref{t:simple} and \ref{t:acyclic}, we see that $G$ is $(\ell+1)$-choosable. Together with the observations above for finite simple graphs and acyclic digraphs, we obtain the following generalisation of Theorems \ref{t:simple} and \ref{t:acyclic}.
\begin{theorem}\label{t:frac}For each $k\geq 2$, any countable graph or countable acyclic digraph is $(1/k)$-majority $(k+1)$-choosable.\end{theorem}
In addition, using the result of \cite{GKP} and \cite{KS18} for finite digraphs, we obtain the following bound for countable digraphs; note, however, that for $k=2$ this gives a weaker bound than that established by Anholcer, Bosek and Grytczuk \cite{ABG20}.
\begin{theorem}For each $k\geq 2$, every countable digraph is $(1/k)$-majority $(2k+1)$-choosable.\end{theorem}
\subsection{Majority correspondence colouring}
In this section we discuss an extension of majority choosability based on correspondence (or DP) colouring. Correspondence colouring is a generalisation of list colouring in which for every edge $uv$ there is a set of forbidden (ordered) colour pairs, with every colour at $u$ (or $v$) being in at most one pair. Choosing these pairs to be the monochromatic ones, we recover list colouring. However, correspondence colouring may require more colours than list colouring: for even cycles with lists of size two, a list colouring exists but a correspondence colouring need not.

We define majority correspondence colourings in the same way. Given a graph $G$, equip each vertex $v$ with a list $L(v)$ of $k\geq 2$ colours, and each edge $uv$ with a set of bad pairs $B_{uv}\subset L(u)\times L(v)$, where for every $c\in L(u)$ there is at most one pair $(c,c')\in B_{uv}$, and for every $c\in L(v)$ there is at most one pair $(c',c)\in B_{uv}$. For a colouring $\chi$ we say an edge $uv$ is bad if $(\chi(u),\chi(v))\in B_{uv}$, and we say $\chi$ is a \textit{$(1/k)$-majority correspondence colouring} if for every vertex $v$, at most a $1/k$ proportion of the edges incident with $v$ are bad. Similarly we may define majority correspondence colourings for digraph with respect to the outedges from each vertex. We say that $G$ is $(1/k)$-majority $\ell$-correspondence colourable if for every collection of lists $\mathcal L=\{L(v)\mid v\in V(G)\}$ of size $\ell$ and bad sets $\mathcal B=\{B_{uv}\mid uv\in E(G)\}$ there is a $(1/k)$-majority correspondence colouring of $G$.

Note that every finite graph or acyclic digraph is $(1/k)$-majority $k$-correspondence colourable, with the same proof in each case. To adapt the proof of Theorem \ref{t:frac} to correspondence colouring, we first need to modify the statement of Lemma \ref{restrict}. 
\begin{lemma}\label{new}Let $V$ be a countable set and suppose that each $v\in V$ has a list $L(v)$ of $\ell+1$ colours. Let $\mathcal X$ be a countable (possibly finite) collection of sets, such that each $X\in\mathcal X$ consists of infinitely many pairs $(v,c)$ with $v\in V$ and $c\in L(v)$, all having distinct first elements. Then there is a choice of $\ell$-element sublists $L'(v)\subset L(v)$ such that for every $X\in\mathcal X$ there are infinitely many pairs $(v,c)\in X$ such that $c\not\in L'(v)$.\end{lemma}
\begin{proof}Fix an ordering of $\mathcal X\times\N$. For each pair $(X,i)$ in turn, choose any pair $(v,c)$ in $X$ such that no pair containing $v$ has previously been chosen, and set $L'(v)=L(v)\setminus\{c\}$. Finally, choose sublists which have not been defined arbitrarily.\end{proof}
Equipped with this version, we may now modify the proof of Theorem \ref{t:simple} to deal with correspondence colouring. To do this, for each pair $(u,c)$ with $u\in V(G)$ and $c\in L(u)$, let $X_{u,c}=\{(v,c'):v\in N(u) \wedge (c,c')\in B_{uv}\}$. Set $\mathcal X$ to be $\{X_{u,c}:\lvert X_{u,c}\rvert=\infty\}$. Note that $\mathcal X$ satisfies the assumptions of Lemma \ref{new}, and so we may choose sublists $\bigl(L'(v)\bigr)_{v\in V(G)}$ in accordance with that lemma. Now define the colouring $\chi$ as in the proof of Theorem \ref{t:simple}. If $v_i$ has finite degree, then choosing an appropriate colouring $\chi_m$ shows that at most the required proportion of edges meeting $v_i$ are bad. If $v_i$ has infinite degree, but $X_{v_i,\chi(v_i)}$ is finite then at most finitely many edges meeting $v_i$ are bad. If $X_{v_i,\chi(v_i)}$ is infinite then by choice of the sublists $v_i$ has infinitely many neighbours $v_j$ such that $v_iv_j$ cannot be bad. The same argument works for acyclic digraphs. Thus we have the following result.
\begin{theorem}For every $k\geq 2$, any countable graph or countable acyclic digraph is $(1/k)$-majority $(k+1)$-correspondence colourable.\end{theorem}
\section*{Acknowledgements}
Research supported by the European Research Council under the European Union's Horizon 2020 research and innovation programme (grant agreement no.\ 639046), and by UK Research and Innovation Future Leaders Fellowship MR/S016325/1. I am grateful to Ant\'onio Gir\~ao for drawing my attention to the question of $(1/k)$-majority colouring.

\end{document}